\documentclass[twocolumn,USenglish]{article}
\usepackage[utf8]{inputenc}
\usepackage[big]{dgruyter}

\usepackage{microtype}

\usepackage{accents}
\usepackage{braket}
\usepackage{bbm}
\usepackage{lmodern}  

\theoremstyle{definition}
\newtheorem{remark}{Remark}
\newtheorem{theorem}{Theorem}
\newtheorem{proposition}{Proposition}
\DeclareMathOperator*{\argmin}{arg\,min}

\newcommand{\ubar}[1]{\underaccent{\bar}{#1}}

\usepackage{xcolor}

\usepackage{caption}
\usepackage{algorithm,algpseudocode}

\begin{document}


  \title{Surrogate Models in Bidirectional Optimization of Coupled Microgrids}
  \runningtitle{Surrogate Models in Bidirectional Optimization}

\subtitle{Ersatzmodelle in bidirektionaler Optimierung gekoppelter Microgrids}

  \author[1]{Manuel Baumann}
  \author[1]{Sara Grundel}
  \author*[2]{Philipp Sauerteig}
  \author[3]{Karl Worthmann}
  \runningauthor{Baumann et al.}
  \affil[1]{Max Planck Institute for Dynamics of Complex Technical Systems, \texttt{\{baumann,grundel\}@mpi-magdeburg.mpg.de}}
  \affil[2]{Technische Universit\"{a}t Ilmenau, \texttt{philipp.sauerteig@tu-ilmenau.de}}
  \affil[3]{Technische Universit\"{a}t Ilmenau, \texttt{karl.worthmann@tu-ilmenau.de}}

\abstract{%
The energy transition entails a rapid uptake of renewable energy sources. Besides physical changes within the grid infrastructure, energy storage devices and their smart operation are key measures to master the resulting challenges like, e.g., a highly fluctuating power generation. For the latter, optimization based control has demonstrated its potential on a microgrid level. However, if a network of coupled microgrids is considered, iterative optimization schemes including several communication rounds are typically used. Here, we propose to replace the optimization on the microgrid level by using surrogate models either derived from radial basis functions or neural networks to avoid this iterative procedure. We prove well-posedness of our approach and demonstrate its efficiency by numerical simulations based on real data provided by an Australian grid operator.
}
  
  \received{05-07-2019}
  \revised{04-09-2019}
  \accepted{06-09-2019}

\maketitle

\hyphenation{Ener-gie-er-zeu-gung Ap-pro-xi-ma-ti-ons-an-satz}

\noindent{\bf Zusammenfassung:} 
Die Energiewende bringt einen raschen Zuwachs eneuerbarer Energiequellen mit sich. Neben den physikalischen Ver\"anderungen der Netzinfrastruktur spielen Energiespeichereinheiten und deren intelligente Nutzung eine entscheidende Rolle, um die sich ergebenden Probleme wie z.B. die stark schwankende Energieerzeugung zu bew\"altigen. In Bezug auf Letztere haben optimierungsbasierte Steuerungstechniken ihr Potential auf Microgrid-Ebene unter Beweis gestellt. Betrachtet man jedoch ein Netzwerk gekoppelter Microgrids, werden \"ublicherweise iterative Optimierungsans\"atze gew\"ahlt, welche mit mehreren Kommunikationsrunden einhergehen. Um derartigen Kommunikationsschleifen vorzubeugen, schlagen wir vor, den Optimierungsschritt auf Microgrid-Ebene durch den Einsatz geeigneter Ersatzmodelle zu vermeiden. Den hier verwendeten Ersatzmodellen liegen zum einen radiale Basisfunktionen und zum anderen neuronale Netze zugrunde. Wir zeigen, dass unser Ansatz wohlgestellt ist und demonstrieren die Effizienz anhand numerischer Simulationen basierend auf realen Daten eines australischen Verteilnetzbetreibers.
\\
\vspace{-0.15cm}

\noindent \textbf{Schlagw\"orter:} 
Smart Grids, Modellprädiktive Regelung, Verteilte Optimierung, Ersatzmodelle, Bidirektionale Optimierung, Neuronale Netze, Radiale Basisfunktionen

\section{Introduction}

The share of renewable energy sources rapidly increases; also due to more and more installed devices like e.g., solar panels at household-level. %
Hence, households become \textit{prosumers}, i.e., power is not only consumed but also produced. 
Therefore, energy generation and distribution takes place in a distributed way. 
In particular, energy can be transmitted bidirectionally between the grid and the prosumers, which results in a paradigm shift in the grid organization. %
In addition, 
prosumers may also possess some kind of energy storage device %
in order to manipulate their power demand profiles by either charging or discharging. %
From the grid operator's perspective it might be beneficial that charging decisions are not made based on local information only. %
Instead taking into account information on the entire grid may improve the system-wide operation, e.g.,\ to flatten %
the overall power demand within the grid in order to facilitate the power supply.~\cite{ParhLotf15}. 
In order to achieve this goal, communication is needed. In the future, each household shall be equipped with a smart meter which yields so-called \emph{smart homes}. Smart meters collect data and communicate with the grid operator automatically. 

A straight-forward way to optimally operate the overall system 
is to formulate one large-scale optimization problem and to solve it in a centralized way, see, e.g.~\cite{OlivCani11}. %
This approach, however, is hard to realize in practice. One of the disadvantages is that some central node needs the complete information about the grid, which is, e.g.\ due to data privacy, not desirable. %
Alternatives are decentralized or distributed optimization algorithms. 
In~\cite{OkubYosh19} the authors propose a decentralized approach to steer energy storage systems in order to avoid over-capacity of pole transformers while maintaining a high charging amount of energy storage systems in low-voltage distribution systems. %
The other option mentioned above are distributed optimization methods such as distributed dual ascent~\cite{BertTsit89}, Alternating Direction Method of Multipliers (ADMM)~\cite{BoydPari11} or %
Augmented Lagrangian based Alternating Direction Inexact Newton (ALADIN)~\cite{HousFras16}. %
These algorithms use a star-shaped communication topology, i.e.\ each smart home communicates only with the grid operator and does not share any information with its neighbours. Nevertheless, in every iteration %
each household has to transmit specific (personal) data to the grid operator, see also~\cite{BrauFaul18} and~\cite{EngeJian19} for an application of ADMM and ALADIN to electrical networked systems, respectively. %
In order to exploit the potential of these algorithms they are typically embedded within a Model Predictive Control (MPC) framework. MPC is a state-of-the-art technique to tackle optimal control problems by solving finite-dimensional optimization problems successively, see e.g.~\cite{GruePann17} for an introduction to MPC and~\cite{KhalSavk10,PariRiko14} for MPC approaches in electrical networks. 

An alternate option to steer the power demand of local agents besides battery control is to schedule so-called controllable loads. Controllable loads can be shifted in time to avoid bottlenecks in the energy supply, see e.g.~\cite{GradSilv15,BrauGrue16}. There is also a large potential in the context of stochastic optimization of smart grids. For weather forecasting methods we refer to~\cite{AppiOrdi18}. %
How to integrate electrical vehicles into the electricity network under uncertainties is described in~\cite{AppiMuno18}.

Considering the power networks described so far, it is assumed that exchange of energy within the grid is possible at any time and does not cause any losses or additional costs, which might (approximately) hold for domestic nets, e.g.\ a %
town. In this paper, we refer to these grids as microgrids (MGs). In~\cite{Lass11,BrauSaue19}, the concept of coupled MGs is used to tackle large-scale problems incorporating several MGs. In the latter, the authors show that even if each single MG is optimally operated, there is still room for improvement if energy can be exchanged among %
MGs. Therefore, a second optimization problem is solved on a higher grid level in order to optimally exchange energy %
resulting in %
a bilevel optimization problem~\cite{SinhMalo17}.

In~\cite{GrunSaue19}, the authors propose to replace the distributed optimization routine on the lower grid level by a surrogate model in order to speed-up the calculation and further reduce communication effort. Here, Radial Basis Functions (RBFs)~\cite{Buh2003} are used to approximate the input-output behaviour of ADMM within the framework of coupled MGs established in~\cite{BrauSaue19}. Besides RBFs there are various methods to learn the behaviour of a complex function. Artificial Neural Networks (NNs) are one of the most popular representatives of modern artificial intelligence techniques and are often used in practice due to their success in various application fields, see e.g.\ the survey article~\cite{AbioJant18}. %
In~\cite{ZhanXu10} the authors forecast loads in a power grid using NNs, whereas in~\cite{SianCeca12} NNs are used in an optimal power flow framework.
The main advantage of using surrogates is that communication effort can be reduced. 

In this paper, we extend the idea of coupled microgrids established in~\cite{BrauSaue19} by proposing an iterative \emph{bi-directional} optimization routine in order to improve the overall performance. 
Due to its iterative structure, however, our method comes along with a strong need for communication between smart homes and grid operator. As a remedy we present two approaches to reduce the communication effort by substituting the optimization on microgrid level via surrogate models. A main difference compared to~\cite{GrunSaue19} lies in the different input-output map that is replaced by the surrogate models, for which we can %
show that %
each input uniquely determines an (optimal) output. %
Furthermore, we also take NNs as potential surrogate models into account 
and study the performance of the resulting approximations numerically in an MPC framework.  %
 Our simulations show that the proposed method approximately recovers the performance based on using ADMM but significantly reduces the communication burden. The effect of applying surrogate models within MPC extends our previous work~\cite{GrunSaue19} where a surrogate model based on RBFs was only applied in a static optimization problem.

The paper is structured as follows: In Section~\ref{sec:Model_coupledMGs} we formulate a mathematical model for coupled microgrids that consists of two hierarchy levels, and introduce optimization problems corresponding to each of them. In the consecutive section, we propose an iterative scheme that requires the solution of a distributed optimization problem on the lower level which is solved using ADMM. In Section~\ref{sec:surrogates}, we investigate the impact of disturbances w.r.t. the lower-level solution on the performance measured in terms of the upper-level objective function. Based on the results, we propose to replace ADMM by surrogates in order to reduce communication effort and computation time. The performance of the optimization scheme incorporating surrogates is analysed in an MPC framework in Section~\ref{sec:results}.

\section{A model for coupled microgrids}\label{sec:Model_coupledMGs}

We consider a system of coupled microgrids (MGs) and call it a \textit{smart grid}. Each MG consists of several residential energy systems (agents) coupled through the grid operator, which can be seen as Central Entity (CE). The coupling of the microgrids is done through a network, where some MGs are connected by a transmission line and others are not connected, cf.~Figure~\ref{fig:coupled MGs}.   
\begin{figure}[ht]
	\centering
	\includegraphics[scale=0.55]{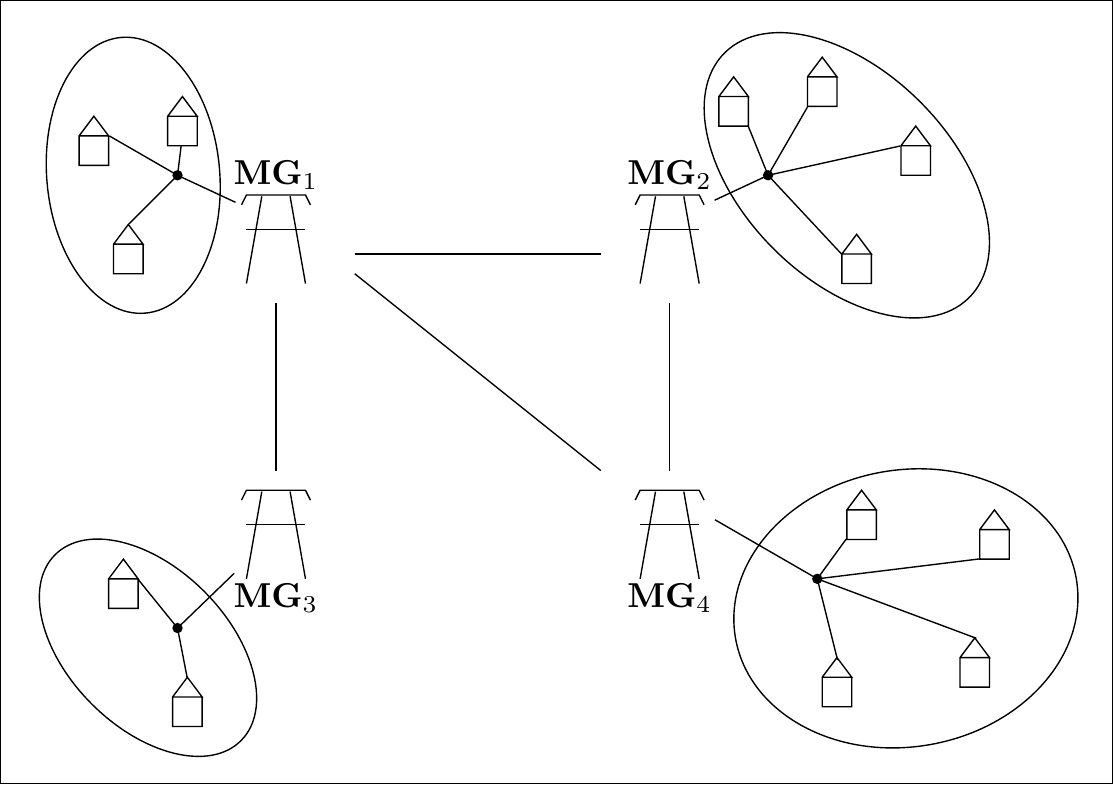}
	\caption{Upper-level model: Schematic representation of four coupled MGs. Energy exchange is possible only via transmission lines.}
	\label{fig:coupled MGs}
\end{figure}

\subsection{Upper-level model: Energy exchange}%
\label{ch:coupled}

We assume that we have $\Xi \in \mathbb{N}$ many MGs which are partially coupled via transmission lines and can be interpreted as nodes of a non-complete graph, see Figure~\ref{fig:coupled MGs} where $\Xi = 4$. %
Each  MG~$\kappa$, for $\kappa \in [1:\Xi]$, consists of $\mathcal{I}_{\kappa}\in \mathbb{N}$ agents modeled in detail in Subsection~\ref{ch:res}.
We assume that each MG $\kappa$ has an average power demand~$\bar{z}_\kappa(n)$  at time~$n$. 
Given this, we can compute the total power demand $\mathcal{I}_{\kappa}\bar{z}_{\kappa}$ of a MG. 
The control goal is to exchange power among the MGs in a way such that a desired quantity $\bar{\zeta}(n)$ is targeted by controlling the residential storage units. We specify~$\bar{\zeta}(n)$ in Subsection~\ref{ch:res}, and assume for the moment that this quantity is known to the grid and has advantages for the grid operation. We assume that this desired quantity is independent of $\kappa$, but this is not necessary for the rest of the discussion. 

Let $\delta_{\nu\kappa}$ describe the percentage of power~$\mathcal{I}_\nu\bar{z}_\nu(n)$ that is transferred from  MG~$\nu$ to  MG~$\kappa$.  We enforce~$\delta_{\nu \kappa}$ equals zero if there is no transmission line between the two  MGs. Otherwise, the power demand of a MG~$\kappa$ is given by its own total power demand $\delta_{\kappa \kappa} \mathcal{I}_{\kappa} \bar{z}_\kappa$, where $\delta_{\kappa \kappa}$ is what remains at the MG, and the sum over the power received from connected MGs, $\sum_{\nu \neq \kappa} \delta_{\nu \kappa} \mathcal{I}_{\nu} \bar{z}_{\nu}$.
For each time step in our prediction horizon of length \mbox{$N \in \mathbb{N}_{\geq 2}$}, we want to match this to the desired power demand starting at time $k$ for $N$ timesteps of each MG in a least-squares sense. The objective function is thus given by \mbox{$\mathcal{J} : \mathbb{R}^{\Xi N} \times \mathbb{R}^{\Xi \times \Xi \times N} \to \mathbb{R}$},
\begin{align}
	 (\bar{z},\delta)  \mapsto  \sum_{n=k}^{k+N-1} 
	 \sum_{\kappa=1}^{\Xi} \left( \bar{\zeta}(n)  \mathcal{I}_\kappa 
	- \sum_{\nu=1}^{\Xi} \delta_{\nu \kappa}(n)  \eta_{\nu \kappa} \mathcal{I}_\nu \bar{z}_\nu(n) 
	\right)^2. \label{OP_upper_cost_function}
\end{align}
Here, the vector $\bar{z} = (\bar{z}(k),\ldots,\bar{z}(k+N-1)^\top)^\top$ with $\bar{z}(\cdot) \in \mathbb{R}^{\Xi}$ stacks the average power demand per MG and time step while the matrix $\eta = (\eta_{\nu \kappa})_{\kappa, \nu =1}^\Xi \in [0,1]^{\Xi \times \Xi}$ incorporates efficiencies along the transmission lines.
	
We are interested in minimizing~\eqref{OP_upper_cost_function} under the following constraints: All exchange rates $\delta_{\nu \kappa}$ are within the interval $[0,1]$, sum up to $1$, and only transfer power in one direction, meaning that either $\delta_{\nu \kappa}$ or $\delta_{\kappa\nu}$ is zero. Moreover, note that at this grid level, the average power demands per MG are known. Following~\cite{BrauSaue19}, the optimization problem of the upper-level is, thus, formulated over the exchange rates~$\delta$,
\begin{subequations}\label{OP_upper}
\begin{align}
	 \min_{\delta \in \Delta} \quad & \mathcal{J}(\bar{z},\delta)  \\
	 \mathrm{s.t.} \quad & \sum_{\kappa=1}^{\Xi} \delta_{\nu \kappa}(n) = 1 \label{Constraint:conservation} \\
	&  \delta_{\nu \kappa}(n) \cdot \delta_{\kappa \nu}(n) \leq 0, \ \kappa \neq \nu\label{Constraint:oneway} \\
	& \forall \, \nu, \kappa \in [1:\Xi], \, n \in [k:k+N-1],  \nonumber
\end{align}
\end{subequations}
where $\delta_{\nu \kappa}(n)$ denotes the power exchange rate from  MG~$\nu$ to  MG~$\kappa$ at time instance~$n$, and \mbox{$\Delta \! := \! \left\{ \delta \! \in \! [0,1]^{\Xi \times \Xi \times N}  \! \mid \! \delta_{\nu \kappa}(n) \! = \! 0, \text{ if no line from }\nu \text{ to }\kappa \right\}$}. Constraints~\eqref{Constraint:conservation} and~\eqref{Constraint:oneway} ensure that the whole energy of each  MG is \emph{scheduled} and that exchanges via transmission lines can only occur in one direction during one time step. 
The definition of the set~$\Delta$ encodes the grid topology and, hence, avoids scheduling energy exchange in the absence of a transmission line, cf.~Figure~\ref{fig:coupled MGs}. 
We denote the feasible set of~\eqref{OP_upper} by
\begin{align}
	\mathbb{D}^\delta \; = \; \Set{\delta \in \Delta \subseteq [0,1]^{\Xi \times \Xi \times N} | \text{\eqref{Constraint:conservation} - \eqref{Constraint:oneway} hold}}. \nonumber
\end{align}
The efficiency of a transmission line does not depend on the direction of the transfer, i.e.\ the matrix $\eta$ is symmetric. Furthermore, we assume no loss without transport, i.e.\ $\eta_{\kappa \kappa} \equiv 1$ for all~$\kappa \in [1:\Xi]$, in the rest of the paper.

\subsection{Lower-level model: Single microgrid}
\label{ch:res}

As we have seen in the previous section, we consider an average power demand at each~MG as well as some desired quantity $\bar{\zeta}(n)$. In order to understand these quantities better, we explain the modeling of the MG in all detail. The basis of our considerations forms the model presented in~\cite{WortKell15} and its extension~\cite{BrauFaul18}. 
 Therefore, let each subsystem be equipped with an energy generation device (e.g. roof top photo-voltaic panels) and some storage device (e.g. a battery). Then, the $i$-th system, $i \in [1:\mathcal{I}_\kappa]$ in MG $\kappa$, can be described by the discrete time system dynamics,
\begin{subequations}\label{SysDyn}
\begin{align}
	x_{\kappa_i}(n+1) \; & = \; \alpha_{\kappa_i} x_{\kappa_i}(n) + T(\beta_{\kappa_i} u_{\kappa_i}^+(n) + u_{\kappa_i}^-(n)) \label{eq:batteryDynamics} \\
	z_{\kappa_i}(n) \; & = \; w_{\kappa_i}(n) + u_{\kappa_i}^+(n) + \gamma_{\kappa_i} u_{\kappa_i}^-(n), \label{eq:system_output}
\end{align}
\end{subequations}
where $x_{\kappa_i}(n)$ and $z_{\kappa_i}(n)$ denote the State of Charge (SoC) of the battery and the power demand at time instance $n \in \mathbb{N}_0$, respectively. The latter incorporates the net consumption, $w_{\kappa_i}(n) = \ell_{\kappa_i}(n)  - g_{\kappa_i}(n) $, as the difference of load and power generation, cf. Figure~\ref{fig:RES-Network}.
\begin{figure}[ht]
\centering
\fbox{
\includegraphics[width=0.35\textwidth]{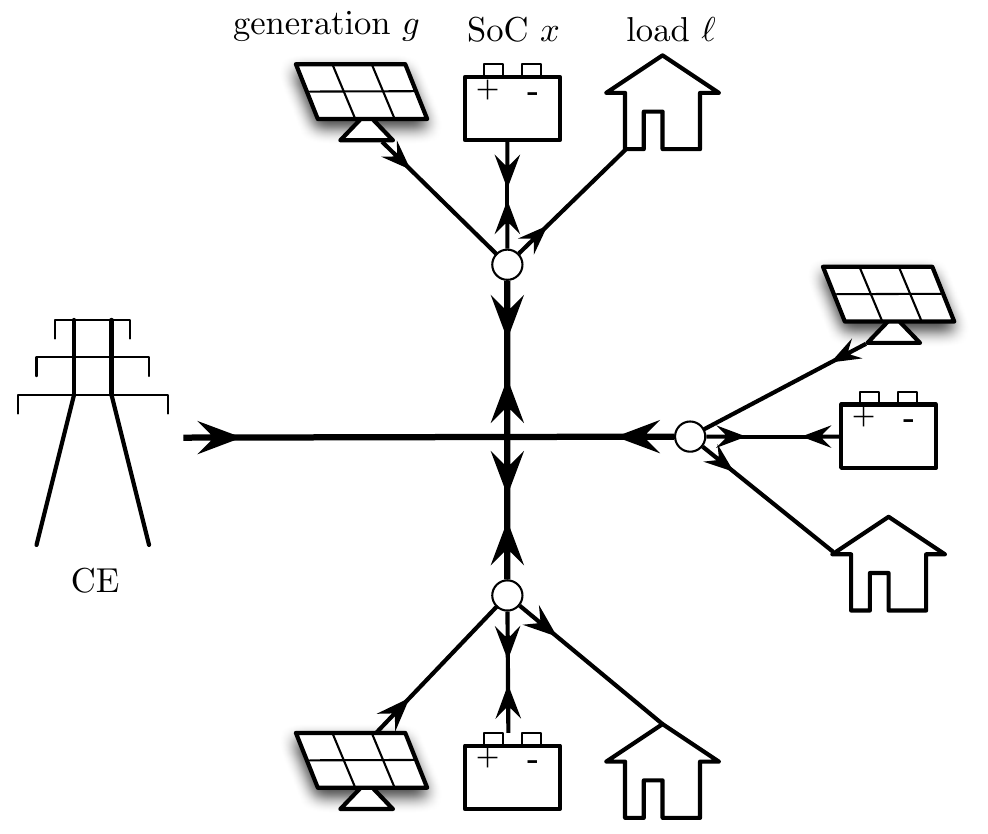}
}
\caption{Lower-level model: Star-shaped network of residential energy systems and central entity (CE). The quantity $w = \ell -g$ in~\eqref{eq:system_output} is obtained from the data set~\cite{RatnWell15}.}
\label{fig:RES-Network}
\end{figure}

The system can be controlled by charging $u_{\kappa_i}^+$ and discharging $u_{\kappa_i}^-$ the battery at each time step. The length of a time step in hours is denoted by \mbox{$T >0$}, e.g. $T = 0.5$ corresponds to a half-hour time window. The constants $\alpha_{\kappa_i},\beta_{\kappa_i},\gamma_{\kappa_i} \in (0,1]$ represent efficiencies w.r.t. self discharge and energy conversion. Furthermore, the initial SoC $x_{\kappa_i}(k) = \hat{x}_{\kappa_i}$, where $k \in \mathbb{N}_0$ denotes the current time instance, is assumed to be known. State and input are subject to the inequality constraints,
\begin{subequations}\label{Constraints}
\begin{eqnarray}
	0 \quad \leq \quad & x_{\kappa_i}(n) & \quad \leq \quad C_{\kappa_i} \\
	\ubar{u}_{\kappa_i} \quad \leq \quad & u_{\kappa_i}^-(n) & \quad \leq \quad 0 \label{Constraint:discharge} \\
	0 \quad \leq \quad & u_{\kappa_i}^+(n) & \quad \leq \quad \bar{u}_{\kappa_i} \label{Constraint:charge}\\
	0 \quad \leq \quad & \frac{u_{\kappa_i}^-(n)}{\ubar{u}_{\kappa_i}} + \frac{u_{\kappa_i}^+(n)}{\bar{u}_{\kappa_i}} & \quad \leq \quad 1. \label{Constraint:chargeNdischarge}
\end{eqnarray}
\end{subequations}
Here, $C_{\kappa_i} \geq 0$ denotes the battery capacity. The last constraint ensures that the bounds on discharging~\eqref{Constraint:discharge} and charging~\eqref{Constraint:charge} hold even if the battery is both discharged and charged during one time step. Note that the case $C_{\kappa_i} = 0$ covers the case, where not all systems have a storage device. Since the future net consumption is not known in advance, it is assumed to be reliably predictable on a short time horizon of size $N$, $N \in \mathbb{N}_{\geq 2}$, time steps. 

For a concise notation we introduce the set $\mathbb{X}_{\kappa_i} = [0,C_{\kappa_i}]$ of feasible states, the set $\mathbb{U}_{\kappa_i} = \Set{(u_{\kappa_i}^+,u_{\kappa_i}^-)^\top \in \mathbb{R}^2 | \eqref{Constraint:discharge} - \eqref{Constraint:chargeNdischarge} \text{ hold}}$ of feasible control pairs and the set 
\begin{align}
	& \mathbb{D}_{\kappa_i} \!=\! \Set{z_{\kappa_i} \!\in\! \mathbb{R}^N| 
	\begin{matrix}
	z_{\kappa_i} = \left(z_{\kappa_i}(k),\ldots,z_{\kappa_i}(k\!+\!N\!-\!1)\right)^\top  \\
		\exists \, u_{\kappa_i} \in \mathbb{U}^N
		\text{such that } \\ 
		x_{\kappa_i}(k) = \hat{x}_{\kappa_i}, \text{\eqref{SysDyn} and \eqref{Constraints} hold}
	\end{matrix}
	} \nonumber
\end{align}
of feasible outputs over the next $N$ time steps, $i \in [1:\mathcal{I}_\kappa]$, $\kappa \in [1:\Xi]$. Referring to the feasible sets of a  MG $\kappa$ we use the Cartesian product, e.g. $\mathbb{D}^{(\kappa)} = \mathbb{D}_{\kappa_1} \times \ldots \times \mathbb{D}_{\kappa_{\mathcal{I}_\kappa}}$ and $z^{(\kappa)} \in \mathbb{D}^{(\kappa)}$. Note that the sets $\mathbb{D}_{\kappa_i}$, $i \in [1:\mathcal{I}_\kappa]$, and hence $\mathbb{D}^{(\kappa)}$ are non-empty, compact, and convex.

The output quantity in~\eqref{eq:system_output} is the power demand~$z_{\kappa_i}$ of an individual agent in MG~$\kappa$. The average power demand $\bar{z}_\kappa =\frac{1}{\mathcal{I}_\kappa} \sum_{i=1}^{\mathcal{I}_\kappa} z_{\kappa_i}$ in each MG can then be computed from the individual power demands, and is used as an input to~\eqref{OP_upper_cost_function}. 
We define $\bar{\zeta}$ in \eqref{OP_upper_cost_function} as a stable reference trajectory by averaging over a past time horizon and over all individual residential units of all microgrids, the so-called overall average net consumption
as proposed in~\cite{BrauFaul18},
\begin{align}
	\bar{\zeta}(n) \, = \,  \frac{1}{\mathcal{I} \cdot \min \{N,n+1\}} \; \sum_{j=n-\min\{n,N-1\}}^n \; \sum_{i=1}^{\mathcal{I}} \; w_{i}(j) \label{eq:ref_val}
\end{align}
where $\mathcal{I} = \sum_{\kappa = 1}^{\Xi} \mathcal{I}_\kappa$ denotes the total number of agents within the entire smart grid. Due to averaging, the trajectory $\bar{\zeta}$ has little fluctuations and yields advantages for the grid operation.

Let us for the moment ignore the coupling described in Subsection~\ref{ch:coupled}. Then, $\delta(n)$ equals the identity for all $n \in [k:k+N-1]$ and $\mathcal{J}$ becomes

\begin{align*}
\bar{z}  \mapsto  \sum_{n=k}^{k+N-1} 
	 \sum_{\kappa=1}^{\Xi} \left( \bar{\zeta}(n)  \mathcal{I}_\kappa 
	-     \mathcal{I}_\kappa \bar{z}_\kappa(n) 
	\right)^2.
\end{align*}
Therefore, the overall objective~$\mathcal{J}$ can be decoupled yielding the local optimization problem 
\begin{align}\label{OP_lower}
	\min_{z^{(\kappa)} \in \mathbb{D}^{(\kappa)}} \quad & g(\bar{z}_\kappa) 
\end{align}
per MG with local objective function $g : \mathbb{R}^N \to \mathbb{R}_{\geq 0}$,
\begin{align}
	g(\bar{z}_\kappa) = \left\| \bar{\zeta}-\bar{z}_\kappa  \right\|_2^2. \nonumber
\end{align}

\subsection{Fully coupled optimization problem}

We are interested in optimizing the function~\eqref{OP_upper_cost_function}. This function, in general, depends on $\delta$ as well as on $\bar{z}$. As seen in the previous section, the average power demand $\bar{z} = \bar{z}(u)$ depends on the control~$u$, which we have to find in such a way that $\mathcal{J}$ is optimal. The overall optimization problem can be written as 
\begin{equation}\label{OP_full}
	 \min_{\delta \in \mathbb{D}^\delta, z^{(\kappa)}\in \mathbb{D}^{(\kappa)}
	 } \quad  \mathcal{J}(\bar{z},\delta).
\end{equation}
Note that due to constraint~\eqref{Constraint:oneway} the optimization of $\mathcal{J}$ w.r.t. $\delta$ is non-convex. Furthermore, the large scale of the optimization w.r.t. $\bar{z}$ causes the use of a centralized solver to be expensive. In addition, using a centralized solver assumes the existence of a global entity gathering the information of the whole grid, in particular the personal data of each household, which is undesirable in practice. Hence, solving~\eqref{OP_full} centralized is impractical. In the subsequent section we present an approach to tackle~\eqref{OP_full} by solving the upper and lower-level problem iteratively. Doing so we avoid a node with full knowledge in the grid and only communicate specific aggregated information in each iteration among the agents.

\section{Bidirectional optimization}

We propose to tackle the optimization problem \eqref{OP_full} in a \textit{bidirectional} way, i.e. we first find an optimal~$\bar{z}$ for~$\delta$ being the identity and then optimize~\eqref{OP_full} w.r.t.~$\delta$ for fixed~$\bar{z}$ in order to find the optimal exchange strategy. This typically already gives a considerable improvement, and has been also done e.g. in~\cite{BrauSaue19}. To solve~\eqref{OP_upper} for a fixed $\bar{z}$ is straight forward, and we use %
a standard Sequential Quadratic Programming (SQP) solver. We refer to~\cite{NoceWrig06} for an introduction to SQP methods. In this paper we show how to incorporate the computed energy exchange from the upper level into the lower-level optimization problem in order to improve the overall performance.

\subsection{Bidirectional optimization scheme}

Assume that each MG $\kappa$, $\kappa \in [1:\Xi]$, within the smart grid has already solved its inherent optimization problem~\eqref{OP_lower} and based on the corresponding solutions~$\bar{z}_\kappa$ an energy exchange policy $\delta^\star$ has been computed according to~\eqref{OP_upper}. 
This exchange yields an updated power demand
\begin{align}
	\bar{z}_\kappa^+(n) = \frac{1}{\mathcal{I}_\kappa} \sum_{\nu=1}^{\Xi} \delta^\star_{\nu \kappa}(n) \, \eta_{\nu \kappa}\, \mathcal{I}_\nu \, \bar{z}_\nu(n) , \ n \in [k:k+N-1], \nonumber 
\end{align} 
and hence, the difference $\Delta\bar{z}_\kappa(n)=\bar{z}_\kappa(n)-\bar{z}_\kappa^+(n)$ in power demand for all MGs. We are interested in updating $\bar{z}$ in such a way that our cost function~\eqref{OP_upper_cost_function},
\begin{align}
	 \mathcal{J}(\bar{z}, \delta^\star) &= \sum_{n=k}^{k+N-1} 
	 \sum_{\kappa=1}^{\Xi} \!\left(\! \bar{\zeta}(n) \, \mathcal{I}_\kappa 
	 \!-\!  \sum_{\nu=1}^{\Xi} \delta^\star_{\nu \kappa}(n) \, \eta_{\nu \kappa}\, \mathcal{I}_\nu \, \bar{z}_\nu(n) 
	 \!\right)^{\!2} \nonumber \\
	& =   \sum_{\kappa = 1}^{\Xi} \mathcal{I}_\kappa^2 \left\| \bar{\zeta} - \bar{z}_\kappa^+ \right\|_2^2=\sum_{\kappa = 1}^{\Xi} \mathcal{I}_\kappa^2 \left\| (\bar{\zeta} +\Delta\bar{z}_\kappa)-\bar{z}_\kappa \right\|_2^2 \nonumber 
\end{align}
is minimized further. 
One could think about fixing $\delta$ and finding an optimal $\bar{z}$. This, however, leads to an optimization problem not avoiding communication and coupling all microgrids. The trick here is now to fix not only the $\delta$ but also the $\bar{z}$-components from all the microgrids but one. 
This leads to optimizing $\mathcal{I}_\kappa^2 \left\| (\bar{\zeta} +\Delta\bar{z}_\kappa)-\bar{z}_\kappa \right\|_2^2$ locally in each MG, where $\Delta \bar{z}_\kappa$ is computed by using the $\bar{z}_\kappa$'s and~$\delta$'s from the previous optimization step.
Intuitively, the difference $\Delta\bar{z}_\kappa$, $\kappa \in [1:\Xi]$, can be interpreted as an additional load or generation, and, therefore, as a change of the desired power demand profile for MG~$\kappa$.  
This yields the modified lower-level optimization problem
\begin{align}\label{OP_lower_modified}
	\min_{z^{(\kappa)} \in \mathbb{D}^{(\kappa)}} \quad g_\kappa(\bar{z}_\kappa) = \left\|  \bar{\zeta}_\kappa^+-\bar{z}_\kappa  \right\|_2^2,
\end{align}
where $\bar{\zeta}^+_\kappa=\bar{\zeta}+\Delta \bar{z}_\kappa$
In this formulation the updated reference trajectories~$\bar{\zeta}_\kappa^+$, $\kappa \in [1:\Xi]$, differ among the single MGs and depend on a given $\delta$ and a given previous~$\bar{z}_\kappa$. The solution of the newly derived lower-level optimization problem can be solved with ADMM for all microgrids independently and in a parallel way.  

Based on the updated reference value we solve~\eqref{OP_lower_modified} and~\eqref{OP_upper} to improve the battery usage and the energy exchange and repeat the optimization until some terminal condition is satisfied, e.g. performance improvement less than  a pre-defined tolerance or maximal number of iterations exceeded. This procedure is summarized in Algorithm~1. %
Note that we only update the reference~$\bar{\zeta}$ on the lower level, since the upper-level optimization problem~\eqref{OP_upper_cost_function} does not change. 
\begin{algorithm}%
{\bf Algorithm~1} Iterative bidirectional optimization scheme \\[-3mm]
\rule{0.49\textwidth}{0.5pt}\\
{\bf Input}: Current time instance $k \in \mathbb{N}_0$, current SoC $x_{\kappa_i}(k) \in \mathbb{X}_{\kappa_i}$, prediction horizon $N \in \mathbb{N}_{\geq 2}$, predicted net consumption $(w_{i_\kappa}(k),\ldots,w_{i_\kappa}(k+N-1))^\top \in \mathbb{R}^N$, ${i_\kappa} \in [1\!:\!\mathcal{I}_\kappa]$, $\kappa \in [1\!:\!\Xi]$, reference trajectory $(\bar{\zeta}(k),\ldots,\bar{\zeta}(k+N-1))^\top \in \mathbb{R}^N$, maximal number $j_{\max} \in \mathbb{N}$ of iterations, and tolerance $\varepsilon > 0$.
\begin{enumerate}
	\item[{\bf Initialization}:]
	\item Set $j = 0$ and $\delta^0(n)= I_\Xi$ for all $n \in [k:k+N-1]$.
	\item \emph{Lower level (parallel in $\kappa$).} Compute $\bar{z}_\kappa^j$ as the solution of~\eqref{OP_lower} using ADMM. 
	\item \emph{Upper level.} Given $\bar{z}_\kappa^j$, solve~\eqref{OP_upper} for $\delta^{j}$ using SQP.
	\item[{\bf While}] $\quad j < j_{\max}$ and $\mathcal{J}(\bar{z}^{j-1},\delta^{j-1}) - \mathcal{J}(\bar{z}^j, \delta^{j}) > \varepsilon$
	\item[{\bf Do}:]
	\item \emph{Lower level (parallel in $\kappa$).} 
	\begin{enumerate}
		\item Compute $\bar{\zeta}_\kappa^+$ from $\bar{z}_\kappa^j$
		and $\delta^j$.
		\item Solve~\eqref{OP_lower_modified} using ADMM and send $\bar{z}_{\kappa}^{j + 1}$ to the upper level.
	\end{enumerate}
	\item \emph{Upper level.} Given $\bar{z}_\kappa^{j+1}$, solve~\eqref{OP_upper} for $\delta^{j + 1}$ using SQP.
	\item $j \to j+1$
\end{enumerate}
\label{alg:opt_scheme}
\end{algorithm}

Neither convergence nor the interpretation of a potential limit of Algorithm~1 is clear a priori. Figure~\ref{fig:OptLoop_convergence}, 
\begin{figure}[ht]
\centering
\includegraphics[width=0.47\textwidth]{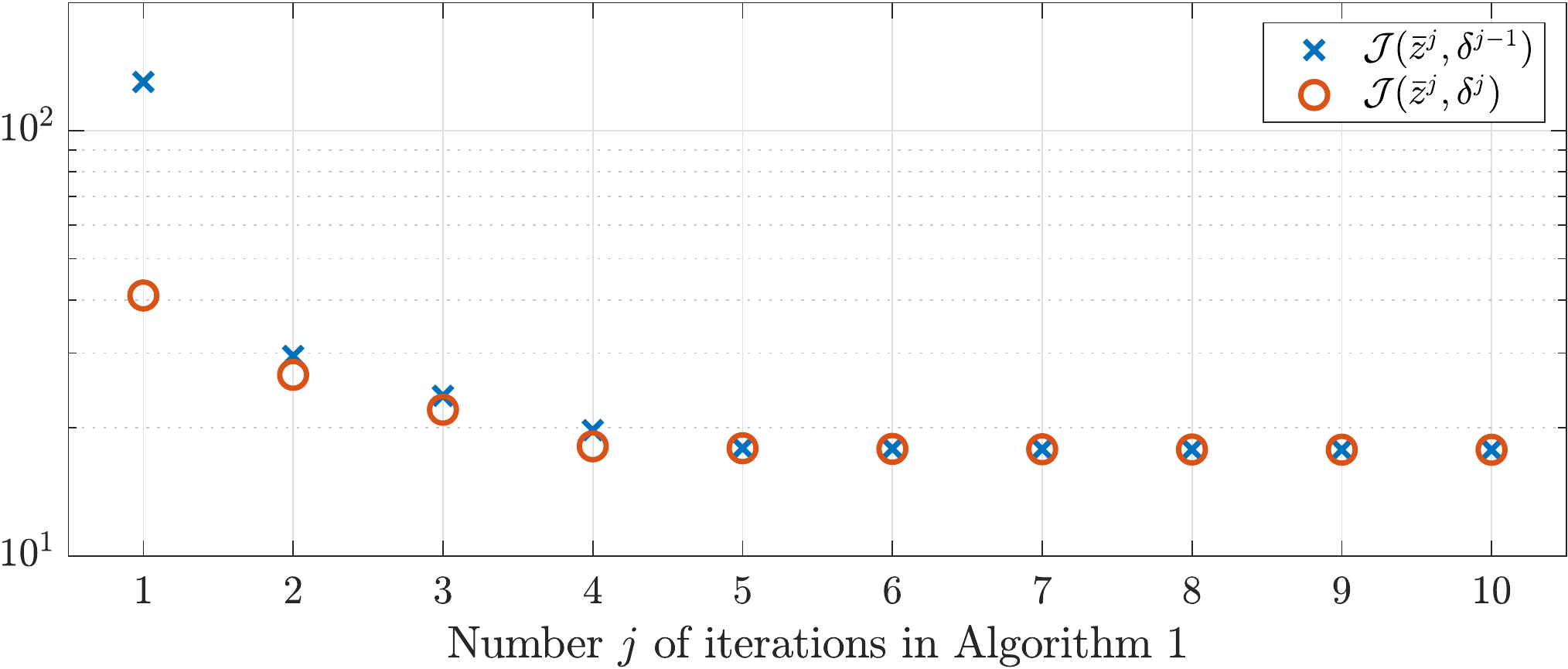}
\caption{Evolution of the costs before and after energy exchange computed according to the bidirectional optimization scheme described in Algorithm~1, i.e. $\mathcal{J}(\bar{z}^j,\delta^{j-1})$ and $\mathcal{J}(\bar{z}^j,\delta^{j})$, resp. Note that $\mathcal{J}(\bar{z}^1, \delta^0)$ yields the costs without microgrid coupling.}
\label{fig:OptLoop_convergence}
\end{figure}
however, experimentally shows convergence of the proposed scheme and a continuous improvement of the upper-level performance index. 
Here, we ran 10 iterations of the optimization scheme and plotted both the objective function values before and after the energy exchange within each iteration. %
The values stagnate after four iterations indicating that additional iterations do not further improve the overall performance. 
The next subsection elaborates on how to solve~\eqref{OP_lower_modified} in a fully distributed way using ADMM.

\begin{remark}
The results displayed in Figure~\ref{fig:OptLoop_convergence} and our numerical investigations indicate convergence to the global minimum of~\eqref{OP_full}. However, a formal (and rigorous) derivation of this conjecture is left for future research. 
\end{remark}

\subsection{Distributed optimization via ADMM}\label{sec:ADMM}

In this section we briefly discuss how to solve the lower-level optimization problem~\eqref{OP_lower} or~\eqref{OP_lower_modified} using an Alternating Direction Method of Multipliers (ADMM) approach. We consider a single MG and therefore omit the index~$\kappa$. Since the averaged output quantity appears in the objective function~\eqref{OP_lower} or~\eqref{OP_lower_modified}, we need to introduce an auxiliary variable~$a$ in order to decouple the lower-level optimization in the following way,
\begin{subequations}\label{OP_ADMM}
\begin{align}
	\min_{z,a} \quad &  g(\bar{a}) = \|\bar{a} - \bar{\zeta}\|_2^2 \\
	\mathrm{s.t.} \quad & \frac{1}{\mathcal{I}} \sum_{i=1}^{\mathcal{I}} a_i - \bar{a} = 0, \quad z_i - a_i = 0 \\
	& z_i \in \mathbb{D}_i \quad \forall \, i \in [1:\mathcal{I}]. \label{OP_ADMM_c}
\end{align}
\end{subequations}
Note that~\eqref{OP_ADMM_c} is a short-hand notation for the battery dynamics~\eqref{SysDyn}-\eqref{Constraints}, and yields a fully decoupled constraint in the variable~$z$.
ADMM is an optimization scheme to solve~\eqref{OP_ADMM} based on the augmented Lagrangian $\mathcal{L}_\rho : \mathbb{R}^{\mathcal{I} N} \times \mathbb{R}^{\mathcal{I} N} \times \mathbb{R}^{\mathcal{I} N} \to \mathbb{R}$,  for $\rho > 0$,
\begin{align}
	\mathcal{L}_\rho(z,a,\lambda) = g(\bar{a}) + \sum_{i=1}^\mathcal{I} \left( \lambda_i^\top (z_i-a_i) + \frac{\rho}{2} \left\| z_i-a_i \right\|_2^2 \right),\nonumber
\end{align}
in a distributed way, cf.~\cite{BoydPari11}. Following~\cite{BrauFaul18}, the ADMM algorithm for~\eqref{OP_ADMM} yields the three-step iteration 
$\ell \mapsto \ell + 1$, 
\begin{subequations}\label{ADMM_scheme}
\begin{align}
	& z_{i}^{\ell + 1} =  \argmin_{z_{i} \in \mathbb{D}_{i}} z_i^\top\lambda_i^\ell + \frac{\rho}{2} \left\| z_{i} - a_{i}^\ell \right\|_2^2 
	\label{ADMM_localOpt} \\
	 & a^{\ell + 1}	=  \argmin_{a \in \mathbb{R}^{\mathcal{I}N}} g(\bar{a}) - \sum_{i=1}^\mathcal{I} a_i^\top \lambda_i^\ell + \frac{\rho}{2}\| z_i^{\ell+1} - a_i \|_2^2 \label{ADMM_scheme_2}\\
	& \lambda_i^{\ell + 1} =  \lambda_i^\ell + \rho (z_i^{\ell + 1} - a_i^{\ell + 1})
\end{align}
\end{subequations}
until some termination condition is satisfied. Note that~\eqref{ADMM_scheme_2} is an unconstrained optimization problem and can be solved explicitly. The problem~\eqref{ADMM_localOpt} can be solved in parallel by each battery in the MG introduced for our model in Section~\ref{ch:res}. Note that scheme~\eqref{ADMM_scheme} assumes communication within the MG, more precisely, each system~$i$, $i \in [1:\mathcal{I}]$, sends its optimal solution $z_i$ to the CE and receives both the updated auxiliary $a_i$ and dual variable $\lambda_i$ in return. The variant discussed in~\cite{BrauFaul18} avoids unnecessary communication overhead by returning a broadcast variable 
instead, which only incorporates information on the aggregated values. 

According to Theorem~3.1 in~\cite{BrauFaul18} the optimization scheme~\eqref{ADMM_scheme} converges in the following sense.

\begin{theorem}\label{thm:ADMM}
Consider Problem~\eqref{OP_ADMM} with $g$ being strictly convex, closed and proper 
and let the iterates $(z^\ell,a^\ell,\lambda^\ell)$ be computed according to~\eqref{ADMM_scheme}. Then the following following statements hold true:
\begin{enumerate}
	\item $(z^\ell - a^\ell)_{\ell \in \mathbb{N}_0}$ converges to zero for $\ell \to \infty$,
	\item $(g(\bar{a}^\ell))_{\ell \in \mathbb{N}_0}$ converges to the optimal value $g^\star$ of~\eqref{OP_ADMM},
	\item $(\lambda^\ell)_{\ell \in \mathbb{N}_0}$ converges to the dual optimal $\lambda^\star$ of~\eqref{OP_ADMM}.
\end{enumerate}
\end{theorem}

According to~\cite{BrauFaul18} and the references~\cite[Section~3]{BoydPari11} and~\cite[Appendix~C]{BertTsit89} therein, problem~\eqref{OP_lower} fulfils the assumptions of Theorem~\ref{thm:ADMM}.

\section{Surrogate models for ADMM}\label{sec:surrogates}
This section is dedicated to surrogate models for the optimization routine~\eqref{ADMM_scheme} within a single MG. For simplicity of notation we omit the index~$\kappa$.

Due to the distributive structure of ADMM, the residential energy systems do not need to share information with their neighbours but only with the CE, see also the star-structure in Figure~\ref{fig:RES-Network}. In each iteration~$\ell$ of ADMM, subsystem~$i$ has to transmit its solution $z_{i}^\ell$ of~\eqref{ADMM_localOpt} to the CE. The optimization scheme presented in Algorithm~1, however, suggests to run ADMM more than once in order to improve the performance w.r.t. to the objective function~\eqref{OP_upper_cost_function}. In order to avoid unnecessary communication, we propose to use surrogate models to approximate the optimization routine~\eqref{ADMM_scheme}. More precisely, we are interested in a function which approximates $\varphi : \mathbb{R}^N \times \mathbb{R}^{\mathcal{I}} \times \mathbb{R}^N \to \mathbb{R}^N$: 
\begin{align}\label{eq:mapping}
	\varphi(\bar{w},x(k),\bar{\zeta}) \quad = \quad \bar{z}, 
\end{align}
for all feasible $(\bar{w},x(k),\bar{\zeta}) \in \mathbb{R}^N \times \mathbb{X} \times \mathbb{R}^N$. Note that we do neither assume knowledge on the local net consumption~$w_i$, $i \in [1:\mathcal{I}]$, nor on the future SoC~$x(n)$, $n > k$.

Figure~\ref{fig:comparisonOL} (top) shows that if the approximation~\eqref{eq:mapping} is sufficiently accurate, the impact on the performance of the  optimization scheme is negligible. 
\begin{figure}[ht]
\centering
\includegraphics[width=0.4\textwidth]{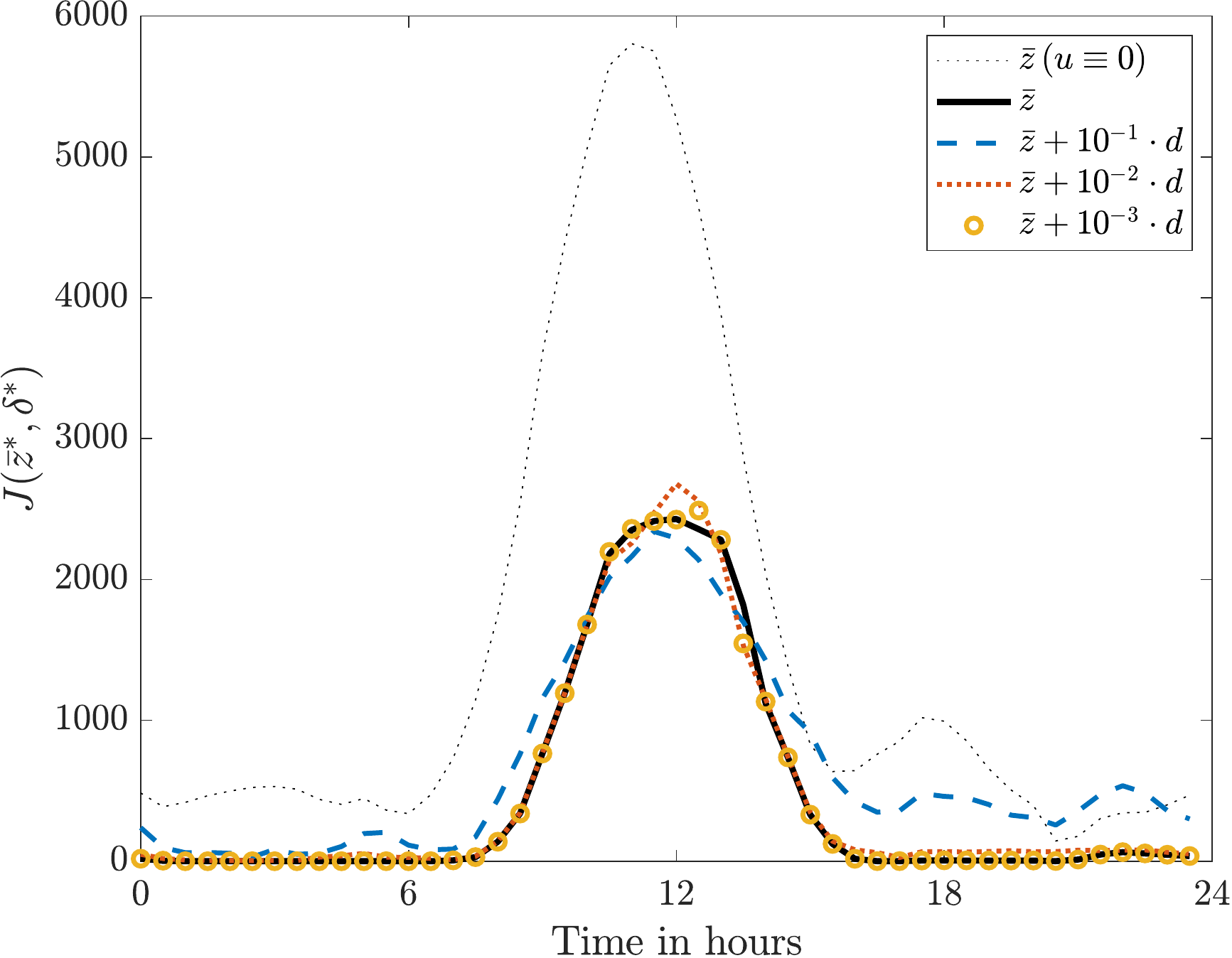}\\
\includegraphics[width=0.4\textwidth]{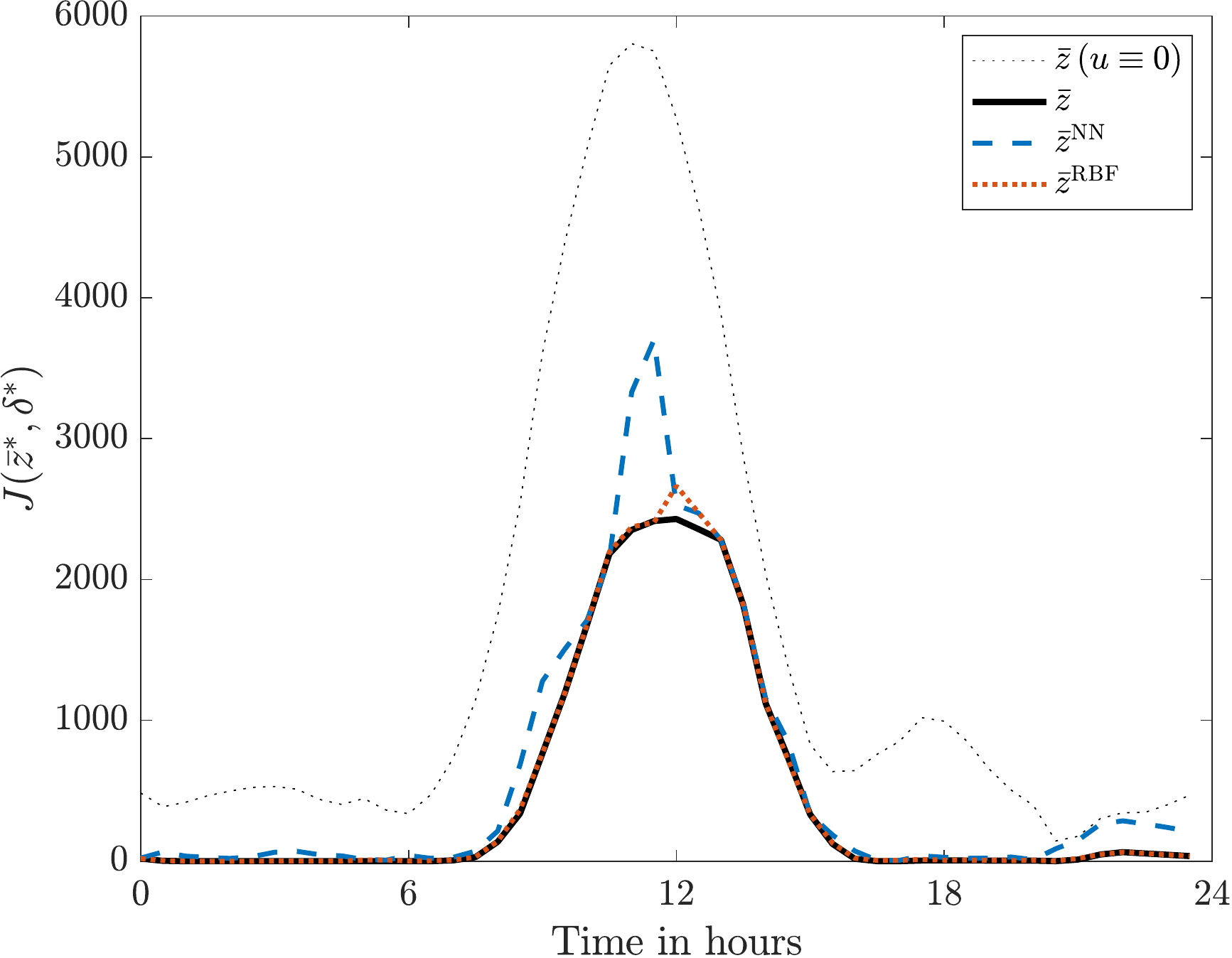}
\caption{Effect of mapping error in~\eqref{eq:mapping} %
(top) and of the approximation via radial basis functions (RBFs) and a neural network (NN) on the open-loop performance $\mathcal{J}(\bar{z},\delta)$ within 48 consecutive time steps (bottom). We use $T = 0.5$h in \eqref{eq:batteryDynamics}.}
\label{fig:comparisonOL}
\end{figure}
Here, the costs $\mathcal{J}(\bar{z},\delta)$ after optimization are visualized for 48 consecutive time steps (equals $24$-hours simulation time). In the experiment, we disturbed the ADMM solution in Algorithm~1 by uniformly distributed additive noise, i. e., $\bar{z} + 10^{-p} \cdot d$, where the vector $d \sim \mathcal{U}(-1,1)$, and $p \in \mathbb{N}_0$ denotes the intensity of the disturbance.%

Note that~\eqref{eq:mapping} might yield approximations to the solution~$\bar{z}$ that either aggravate the performance w.r.t.~\eqref{OP_lower} compared to the solution~$\bar{w}$ associated with $u \equiv 0$ or  solutions that correspond to an infeasible control $\hat{u} \notin \mathbb{U}$. As a remedy, we propose to apply ADMM once after replacing it by a surrogate in the  optimization scheme. More precisely, first we run Algorithm~1 using a surrogate in Step~4(b) until the while loop terminates and then we additionally repeat Steps~4 and~5 using ADMM.

\subsection{Well-posedness}

The following proposition states that for equality in~\eqref{eq:mapping}, a proper mapping is defined. For a concise notation we replace the index $\kappa_i$ by $i$ here.
\begin{proposition}
\label{prop1}
Consider $\varphi$ given by~\eqref{eq:mapping}, where $\bar{z}$ describes the optimal solution of~\eqref{OP_lower} computed via ADMM, i.e. $\bar{z} = \bar{z}(u^\star)$. We assume all hyper-parameter to be fixed meaning that $\{T, \alpha_i, \beta_i, \gamma_i, C_i, \bar{u}_i, \ubar{u}_i\}$ in \eqref{SysDyn}-\eqref{Constraints} are constant over time for all $i \in [1:\mathcal{I}]$. Then $\varphi$ is a mapping, i.e. for all $(\bar{w},x(k),\bar{\zeta}) \in \mathbb{R}^N \times \mathbb{X} \times \mathbb{R}^N$, there exists a uniquely determined $\bar{z} \in \mathbb{R}^N$ such that $\bar{z}$ is the solution to the optimization problem~\eqref{OP_lower}.
\end{proposition}

\begin{proof}
First note that ADMM yields the unique solution of~\eqref{OP_lower}, see e.g.~\cite{BoydPari11}. Furthermore, there are no constraints on $z_i$, $i \in [1:\mathcal{I}]$, and the future SoC can be interpreted as an affine function of the current SoC and the future (dis-) charging rate. Hence, expansion of~\eqref{eq:batteryDynamics} and averaging of~\eqref{eq:system_output} yield,
\begin{align}
	& \min_u \quad  \left\| \bar{z}(u) - \bar{\zeta} \right\|_2^2, \quad \text{ subject to}\nonumber \\
	& x_i(k\!+\!1\!+\!n) =  \alpha_i^{n+1} x_i(k) \!+\! T \! \sum_{\ell=k}^{k+n} \!\alpha_i^{n+k-\ell} \big(\beta_i u_i^+(\ell) \!+\! u_i^-(\ell)\!\big), \nonumber \\
	& x_i(k) = \hat{x}_i, \ \text{ and constraints~\eqref{Constraints}}, \nonumber \\
	& \bar{z}(k+n) = \bar{w}(k+n) + \bar{u}^+(k+n) + \bar{\gamma} \bar{u}^-(k+n), \nonumber \\ 
	& n \in [0:N-1], \nonumber
\end{align}
where $\bar{\cdot}$ denotes the corresponding average value w.r.t. all subsystems, in particular $\bar{z}(n) = \frac{1}{\mathcal{I}} \sum_{i=1}^\mathcal{I} z_i(n)$. This representation of~\eqref{OP_lower} illustrates that the (predicted) average values~$\bar{\zeta} = (\bar{\zeta}(k), \ldots, \bar{\zeta}(k+N-1))^\top$ and~$\bar{w} = (\bar{w}(k), \ldots, \bar{w}(k+N-1))$ and the current SoC~$\left\{x_i(k)\right\}_{i=1}^\mathcal{I}$, uniquely determine the optimal solution~$\bar{z}(u^\star)$ obtained by ADMM.
\end{proof}

\begin{remark}
Note that $\bar{\zeta} \mapsto \bar{z}$ as introduced in~\cite{GrunSaue19} does not define a mapping since the solution~$\bar{z}$ of~\eqref{OP_lower} not only depends on the reference value~$\bar{\zeta}$ but also on the future net consumption~$\bar{w}$ and the current SoC~$x_{i}(k)$.
\end{remark}

\subsection{Radial basis functions approximation}

Radial Basis Functions (RBFs) are used to interpolate functions based on a set of sampling data. We briefly recap some basics on RBFs. For a detailed introduction to theory and application see e.g.~\cite{Buh2003}, for a similar approach where RBFs are used to replace ADMM we refer to~\cite{GrunSaue19}.

Let $M \in \mathbb{N}$ denote the number of samples. Then, the interpolation function of~\eqref{eq:mapping} is given as the sum of basis functions $\psi_m : \mathbb{R}^N \times \mathbb{X} \times \mathbb{R}^N \to \mathbb{R}$, $m \in [1:M]$, and a regularization term $q : \mathbb{R}^N \times \mathbb{X} \times \mathbb{R}^N \to \mathbb{R}^N$. More precisely,
\begin{align}
	\bar{z} \approx \varphi^{\mathrm{RBF}}(\chi) = \sum_{m=1}^M \psi_m(\chi) \alpha_m + q(\chi), \label{eq:rbf_approx} 
\end{align}	
where $\chi = (\bar{w},x(k),\bar{\zeta})$ is the joint inputs of Proposition~\ref{prop1}, and $\alpha_m \in \mathbb{R}^N$, $m \in [1:M]$. The basis functions are %
so-called radial basis functions of the form, $\psi_m(\chi) = \psi(\left\| \chi - \chi_m \right\|)$, where the kernel $\psi$ yields support close to the sampling data $\chi_m$, $m \in [1:M]$. We choose an affine linear regularization $q(\chi) = \beta_0 + B\chi$. Note that different choices are possible. The missing parameters $\alpha_m$, $\beta_0$ and~$B$ are determined by interpolation conditions, cf.~\cite{GrunSaue19,Buh2003}.

In Figure~\ref{fig:RBF_NN_fit}, a possible fit via RBFs is visualized. Here, we interpolated given data from two-weeks of optimization ($4540$ data points) based on sampling data picking each \mbox{$25$-th} data point to train~\eqref{eq:rbf_approx}. Then, we tested~$\varphi^{\mathrm{RBF}}$ on the following day, and plotted the fitting.
\begin{figure}[ht]
\centering
\includegraphics[width=0.4\textwidth]{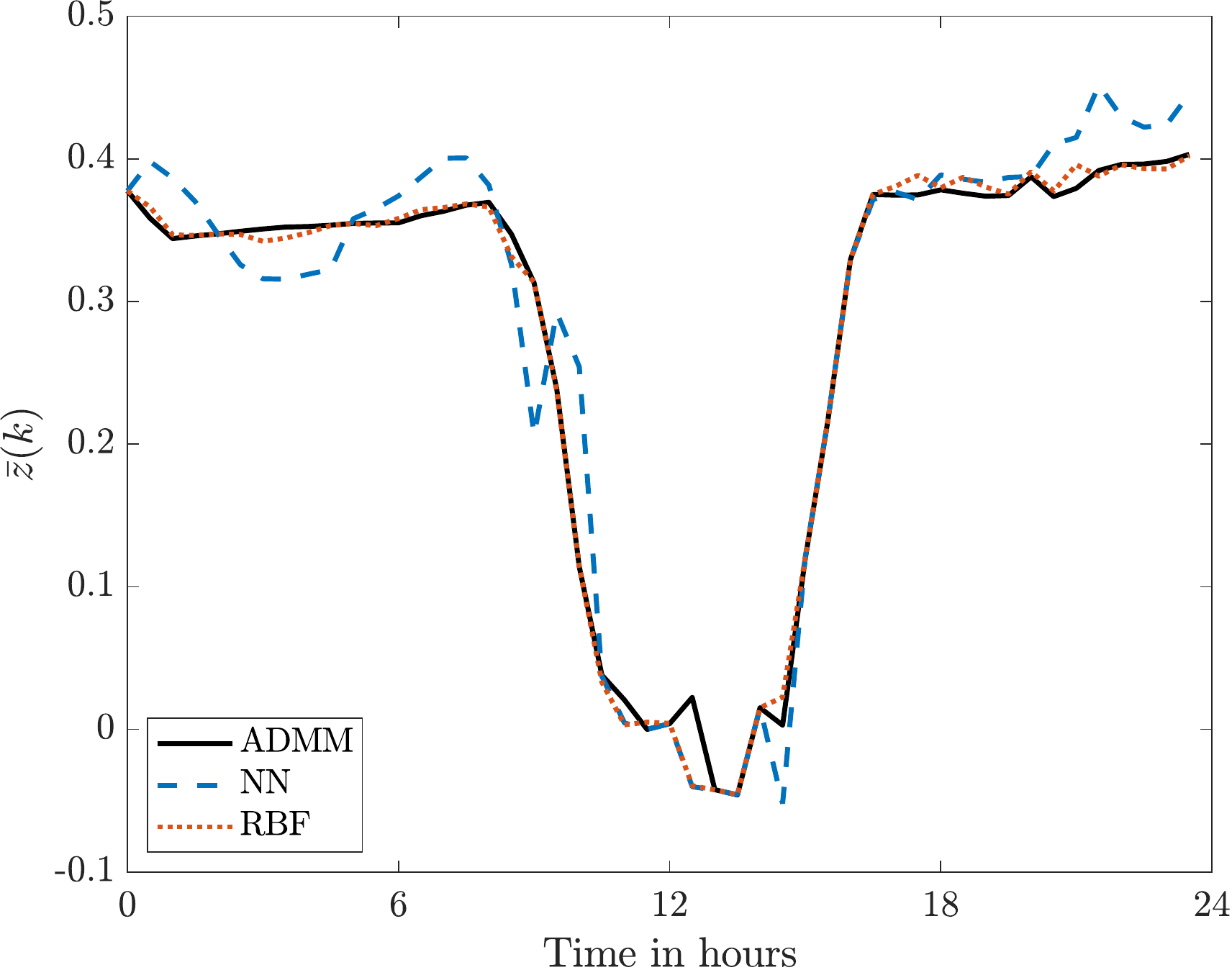} 
\caption{RBF and neural net fitting of the first component $\bar{z}(k)$ of $\bar{z} = (\bar{z}(k),\ldots,\bar{z}(k+N-1))^\top$ within 48 consecutive time steps indicating the quality of the approximation~\eqref{eq:mapping}.}
\label{fig:RBF_NN_fit}
\end{figure}
Our implementation is based on the \texttt{Matlab} toolbox \texttt{DACE}~\cite{LopNS02}. Note that the evaluation time of the RBF approximation grows with the number of data points used. Already with 180 data points to train~\eqref{eq:rbf_approx} with $N=6$ causes the function evaluation of~$\varphi^{\mathrm{RBF}}$ to be expensive, see Table~\ref{tab:communication}. Using more data points would no longer yield an advantage over using ADMM w.r.t. computation time.

\subsection{Neural networks approximation}

Neural Networks (NNs) are a state-of-the-art method in artificial intelligence frameworks. Based on huge amounts of data $M \gg 1$ they are able to learn and recognize patterns in complex systems.
We consider a NN of $l$-layers as an approximation to the mapping~\eqref{eq:mapping}, i.e.,
\begin{align}
	& \bar{z} \approx \varphi^{\mathrm{NN}}(\chi) = \sigma\left(W^{[l]} \ ... \ \sigma(W^{[2]}\chi + b^{[2]}) \ ... \  + b^{[l]}\right),\label{eq:nn_approx}
\end{align}
where $\sigma$ denotes the sigmoid function, and the weights~$W^{[l]}$ and biases~$b^{[l]}$ are determined during the training phase. Here, the number of neurons at layer~$l-1$ and at layer~$l$ determine the number of rows and columns of $W^{[l]}$, respectively. Note that a separate neural network is trained for each MG. For an introduction to deep learning and neural networks we refer the reader to \cite{GoodBeng16,HighHigh18}. To train the NNs, we used \texttt{Matlab}'s built-in toolbox \texttt{nftool}.

The overall goal of the approximation~\eqref{eq:nn_approx} is to be sufficient in the sense of the MPC performance shown in Figure~\ref{fig:comparisonCL}. Our experiments in Figures~\ref{fig:comparisonOL} (bottom) and Figure~\ref{fig:RBF_NN_fit} show that with one hidden layer of ten neurons only, a satisfying approximation on a 24-hours time window can be achieved if the training data is large enough. Note that NNs benefit from big data. In our case study, we trained the NN only on data corresponding to two weeks.

\section{Numerical proof-of-concept}\label{sec:results}

Model Predictive Control (MPC) is a method to tackle optimal control problems on an infinite time horizon by solving a series of finite dimensional optimization problems instead, see e.g.~\cite{GruePann17} for an introduction to non-linear MPC.

\subsection{Model predictive control (MPC)}

Consider the optimal control problem~\eqref{OP_lower}. In order to provide an optimal control sequence over an arbitrary long time horizon we use MPC. To this end, at current time instance $k \in \mathbb{N}_0$ we assume the future net consumption $(w_i(k),w_i(k+1),\ldots,w_i(k+N-1))^\top \in \mathbb{R}^N$ to be predicted for all subsystems $i \in [1:\mathcal{I}]$. Based on the prediction, Algorithm~1 is executed (per MG) to determine control sequences $u_i$, $i \in [1:\mathcal{I}]$, and an exchange strategy $\delta$. Then, only the first instances $u_i(k)$ and $\delta(k)$ are implemented and the time instance is incremented. Algorithm~2 outlines this MPC scheme. 
\begin{algorithm}
{\bf Algorithm~2} MPC for coupled MGs \\[-3mm]
\rule{0.49\textwidth}{0.5pt}\\
{\bf Input}: Current time instance $k \in \mathbb{N}_0$, current SoC $x_{i}(k) \in \mathbb{X}_{\kappa_i}$, prediction horizon $N \in \mathbb{N}_{\geq 2}$, and reference trajectory $(\bar{\zeta}(k),\ldots,\bar{\zeta}(k+N-1))^\top \in \mathbb{R}^N$.
\begin{enumerate}
	\item[{\bf Repeat}:]
	\item Measure current state $x_i(k)$ and update the forecast $(w_i(k),w_i(k+1),\ldots,w_i(k+N-1))^\top$, $i \in [1:\mathcal{I}]$.
	\item Run Algorithm~1 for all MG to get optimal control sequences $u_i^\star = (u_i^\star(k),\ldots,u_i^\star(k+N-1))^\top$ for all subsystems $i \in [1:\mathcal{I}]$, and an optimal exchange strategy $\delta^\star = (\delta^\star(k),\ldots,\delta^\star(k+N-1))$.
	\item Implement $u_i^\star(k)$, $i \in [1:\mathcal{I}]$, and $\delta^\star(k)$ and shift the time instance $k \to k+1$.
\end{enumerate}
\end{algorithm}

Note that Problems~\eqref{OP_lower} or~\eqref{OP_lower_modified} and~\eqref{OP_upper} have to be solved in order to determine $\bar{z}^\star$ and $\delta^\star$ in each MPC iteration. Therefore, the open-loop costs $\mathcal{J}(\bar{z}^{\star},\delta^\star)$ can be computed in each iteration as well (cf. Figure~\ref{fig:comparisonOL}). However, since Step~3 in Algorithm~2 suggests to only implement the first instance of the controls computed in Step~2, these costs are not attained. Instead the stage costs 
\begin{align}\label{eq:CostsCL}
	\sum_{\kappa=1}^{\Xi} \left(\bar{\zeta}(k) \mathcal{I}_\kappa - \sum_{\nu=1}^{\Xi} \delta_{\nu \kappa}^\star(k) \eta_{\nu \kappa} \mathcal{I}_\nu \bar{z}_\nu^\star(k)\right)^2
\end{align}
are realized at each time step $k \in \mathbb{N}_0$ (cf. Figure~\ref{fig:comparisonCL}).

\subsection{Usage of surrogate models in MPC}

We compare the performances using ADMM, RBFs, and NNs on the lower-level, i.e. in Step~4(b) of Algorithnm~1. 
In all numerical simulations we set $T = 0.5$, $N = 6$, $\Xi = 4$, $\mathcal{I}_1 = 50$, and $\mathcal{I}_2 = \mathcal{I}_3 = \mathcal{I}_4 = 10$.\footnote{Note that this setting yields the global optimization problem~\eqref{OP_full} with more than 1000 variables.} The battery parameters were randomly chosen with mean values $C = 0.98$, $\ubar{u} = -0.24$, and $\bar{u} = 0.25$. Based on the battery capacities we set $\hat{x}_i = 0.5 C_i$. In order to incorporate losses along the transmission lines, we used the efficiency matrix,
\begin{align}
	\eta = 
	\begin{bmatrix}
	1.0   &  0.9  &  0.9  &  0.85 \\
	0.9   &  1.0  &  0.0  &  0.85 \\
	0.9   &  0.0  &  1.0  &  0.0 \\
	0.85 & 0.85 &  0.0  &  1.0
	\end{bmatrix} \text{ in } \eqref{OP_upper_cost_function}. \nonumber
\end{align}
For simplicity of the numerical computation, we only replaced the lower-level optimization routine for MG~1 and, thus, avoid training a separate surrogate model for each MG. %
We used \texttt{Matlab} for implementation. 

Results on the MPC closed loop can be found in Figure~\ref{fig:comparisonCL} and Table~\ref{tab:communication}.
\begin{figure}[ht]
\centering
\includegraphics[width=0.4\textwidth]{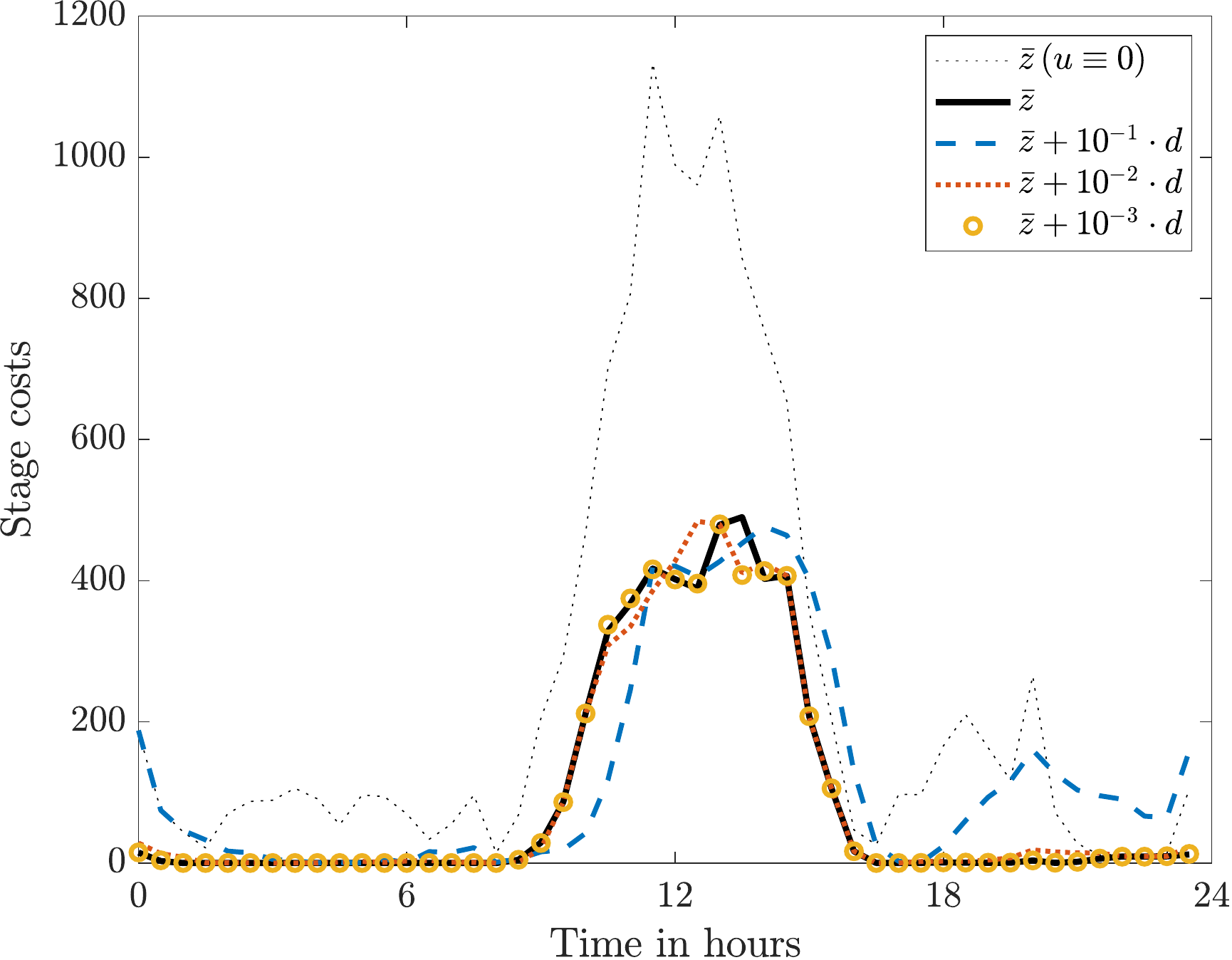}\\
\includegraphics[width=0.4\textwidth]{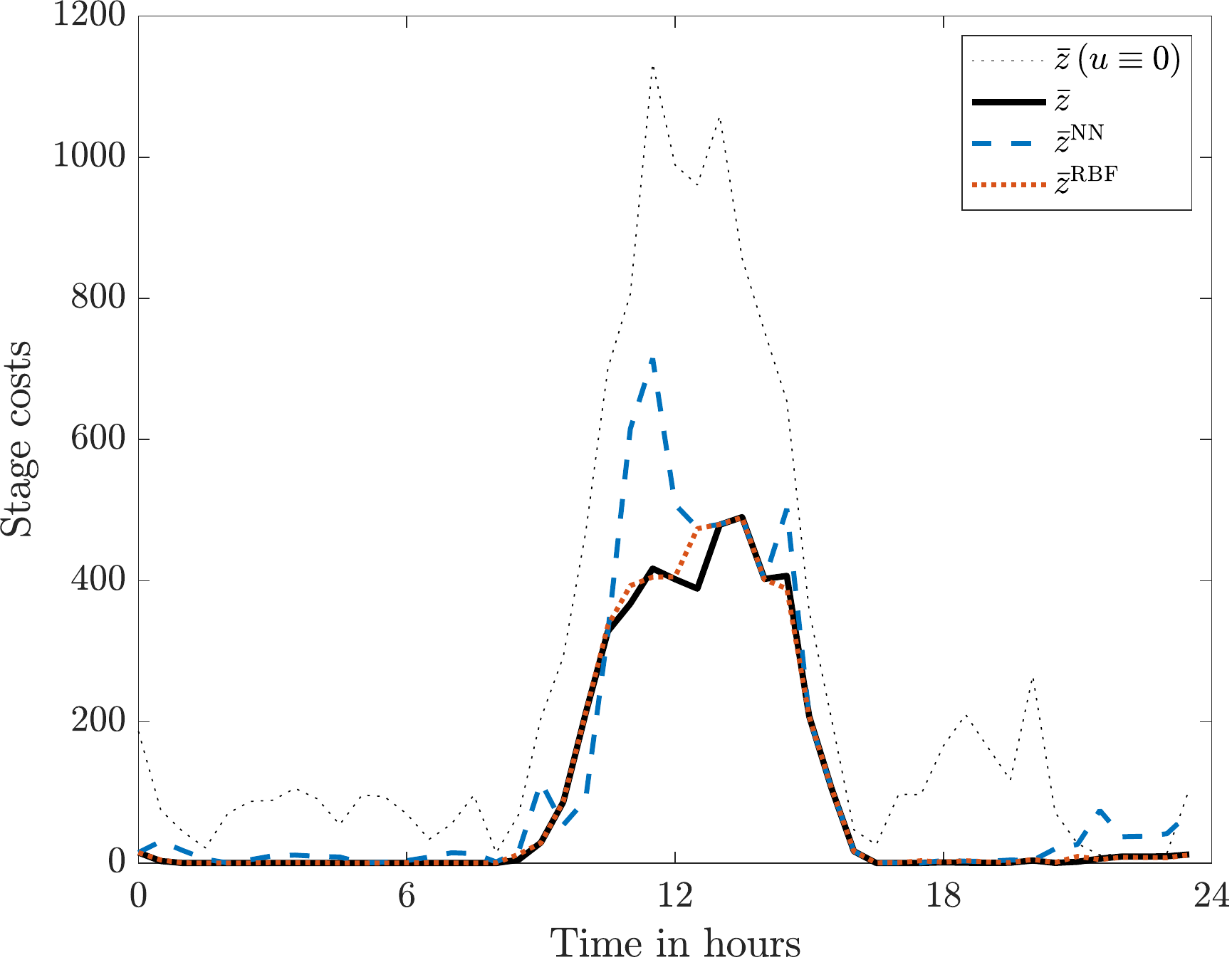}
\caption{Impact of mapping error (top) and approximation via RBF and NN (bottom) on the stage costs~\eqref{eq:CostsCL} within 48 consecutive time steps.}
\label{fig:comparisonCL}
\end{figure}
 \begin{table}[ht]
 \centering
 \begin{tabular}{c|ccc}
  & closed-loop cost & runtime [ms]  \\ 
 \hline 
 no control & 12,228 & ---  \\
 ADMM & 4,416 & 2.5  \\
 RBFs & 4,529 & 1.2  \\ 
 NNs & 5,598 & 0.05  \\ 
 \end{tabular} 
 \caption{Comparison of the summed MPC closed-loop performance $\sum_{k=0}^{47} \sum_{\kappa=1}^{\Xi} \left(\bar{\zeta}(k) \mathcal{I}_\kappa - \sum_{\nu=1}^{\Xi} \delta_{\nu \kappa}^\star(k) \eta_{\nu \kappa} \mathcal{I}_\nu \bar{z}_\nu^\star(k)\right)^2$ and runtime (per call): ADMM vs. RBFs vs. NNs.}
 \label{tab:communication}
 \end{table}
In Figure~\ref{fig:comparisonCL} the closed-loop performances of ADMM (black line) compared to perturbed ADMM, and ADMM (black line) compared to the two surrogate models are visualized. Similar to the open-loop case, small disturbances in ADMM have little impact and RBFs outperform the NN. The first column of Table~\ref{tab:communication} compares the sum of all MPC closed-loop performances using ADMM, RBFs and a NN while in column~2 the average runtimes of these approaches are reported. Note that when using a surrogate, we call ADMM once per MPC iteration. As elaborated in~\cite{BrauFaul18} in each ADMM iteration an $N$-dimensional vector has to be transmitted twice. Hence, both surrogates reduce the need for communication. Two great advantages of ADMM are that the local optimization~\eqref{ADMM_localOpt} can be parallelized and the global optimization is independent of the size of the MG. However, a single function evaluation such as~\eqref{eq:rbf_approx} or~\eqref{eq:nn_approx} is faster than running the entire ADMM optimization routine.

Note that in column~2 of Table~\ref{tab:communication} we ignored the communication between smart homes and CE which is needed to apply ADMM in practice. However, the runtime of ADMM impairs when executed in an actual smart grid while surrogates do not require additional communication.

In order to improve the performance of the NN, more sampling data has to be generated to increase the training set significantly. To avoid large offline computation times, we chose $N=6$, i.e. a prediction horizon of three hours, which is rather short compared to~\cite{BrauFaul18,GrunSaue19}.

\begin{remark}
We point out two implementation details to solve~\eqref{OP_upper} efficiently. First, note that the optimization~\eqref{OP_upper} can be parallelized in~$n$, since there is no coupling. Furthermore, we replace~\eqref{Constraint:oneway} by
\begin{align}
	\delta_{\kappa \nu}(n) \cdot \delta_{\nu \kappa}(n) \; \leq \; \varepsilon \nonumber
\end{align}
for some tolerance $\varepsilon > 0$ to smooth the feasible set. 
\end{remark}

\section{Conclusions}\label{sec:conclusions}
In this paper we recalled an  optimization problem arising in large-scale electrical networks. 
We proposed an iterative \emph{bidirectional} optimization scheme to tackle this problem in a distributed way, and showed numerically that a small error on the lower level does not have noticeable impact on the performance. 
Based on this observation, we replaced the lower-level optimization by surrogate models using radial basis functions and artificial neural networks. 
The numerical results show the potential of using these surrogates to reduce communication effort and computational time in MPC while preserving the overall performance.\\

\vspace{0.5cm}
\noindent \textbf{Funding:}
The authors gratefully acknowledge funding by the Federal Ministry for Education and Research (BMBF; grants 05M18EVA and 05M18SIA). Karl Worthmann is indebted to the German Research Foundation (DFG; grant WO\,2056/6-1).

\bibliographystyle{unsrt}
\bibliography{AT_19}

\end{document}